\documentclass[11pt]{article}
\usepackage[margin=1in]{geometry} 
\usepackage{amssymb,amsfonts,amsmath,bbm,mathrsfs,stmaryrd,mathtools}
\usepackage{xcolor}
\usepackage{url}

\usepackage{extarrows}

\usepackage[shortlabels]{enumitem}
\usepackage{tensor}

\usepackage{xr}
\usepackage[T1]{fontenc}

\usepackage[colorlinks,
linkcolor=black!75!red,
citecolor=blue,
pdftitle={},
pdfproducer={pdfLaTeX},
pdfpagemode=None,
bookmarksopen=true,
bookmarksnumbered=true,
backref=page]{hyperref}

\usepackage{tikz}
\usetikzlibrary{arrows,calc,decorations.pathreplacing,decorations.markings,intersections,shapes.geometric,through,fit,shapes.symbols,positioning,decorations.pathmorphing}

\usepackage{braket}

\usepackage[amsmath,thmmarks,hyperref]{ntheorem}
\usepackage{cleveref}

\creflabelformat{enumi}{#2#1#3}

\crefname{section}{Section}{Sections}
\crefformat{section}{#2Section~#1#3} 
\Crefformat{section}{#2Section~#1#3} 

\crefname{subsection}{\S}{\S\S}
\AtBeginDocument{%
  \crefformat{subsection}{#2\S#1#3}%
  \Crefformat{subsection}{#2\S#1#3}% 
}

\crefname{subsubsection}{\S}{\S\S}
\AtBeginDocument{%
  \crefformat{subsubsection}{#2\S#1#3}%
  \Crefformat{subsubsection}{#2\S#1#3}% 
}

\theoremstyle{plain}

\newtheorem{lemma}{Lemma}[section]

\newtheorem{corollary}[lemma]{Corollary}
\newtheorem{theorem}[lemma]{Theorem}

\theoremstyle{plain}
\theoremnumbering{Alph}
\newtheorem{theoremN}{Theorem}

\theoremstyle{plain}
\theorembodyfont{\upshape}
\theoremsymbol{\ensuremath{\blacklozenge}}

\newtheorem{definition}[lemma]{Definition}

\newtheorem{examples}[lemma]{Examples}
\newtheorem{remark}[lemma]{Remark}
\newtheorem{remarks}[lemma]{Remarks}

\newtheorem{notation}[lemma]{Notation}

\crefname{definition}{definition}{definitions}
\crefformat{definition}{#2definition~#1#3} 
\Crefformat{definition}{#2Definition~#1#3} 

\crefname{ex}{example}{examples}
\crefformat{example}{#2example~#1#3} 
\Crefformat{example}{#2Example~#1#3} 

\crefname{exs}{example}{examples}
\crefformat{examples}{#2example~#1#3} 
\Crefformat{examples}{#2Example~#1#3} 

\crefname{remark}{remark}{remarks}
\crefformat{remark}{#2remark~#1#3} 
\Crefformat{remark}{#2Remark~#1#3} 

\crefname{remarks}{remark}{remarks}
\crefformat{remarks}{#2remark~#1#3} 
\Crefformat{remarks}{#2Remark~#1#3} 

\crefname{convention}{convention}{conventions}
\crefformat{convention}{#2convention~#1#3} 
\Crefformat{convention}{#2Convention~#1#3} 

\crefname{notation}{notation}{notations}
\crefformat{notation}{#2notation~#1#3} 
\Crefformat{notation}{#2Notation~#1#3} 

\crefname{table}{table}{tables}
\crefformat{table}{#2table~#1#3} 
\Crefformat{table}{#2Table~#1#3}

\crefname{lemma}{lemma}{lemmas}
\crefformat{lemma}{#2lemma~#1#3} 
\Crefformat{lemma}{#2Lemma~#1#3} 

\crefname{proposition}{proposition}{propositions}
\crefformat{proposition}{#2proposition~#1#3} 
\Crefformat{proposition}{#2Proposition~#1#3} 

\crefname{propositionN}{proposition}{propositions}
\crefformat{propositionN}{#2proposition~#1#3} 
\Crefformat{propositionN}{#2Proposition~#1#3} 

\crefname{corollary}{corollary}{corollaries}
\crefformat{corollary}{#2corollary~#1#3} 
\Crefformat{corollary}{#2Corollary~#1#3} 

\crefname{corollaryN}{corollary}{corollaries}
\crefformat{corollaryN}{#2corollary~#1#3} 
\Crefformat{corollaryN}{#2Corollary~#1#3} 

\crefname{theorem}{theorem}{theorems}
\crefformat{theorem}{#2theorem~#1#3} 
\Crefformat{theorem}{#2Theorem~#1#3} 

\crefname{theoremN}{theorem}{theorems}
\crefformat{theoremN}{#2theorem~#1#3} 
\Crefformat{theoremN}{#2Theorem~#1#3} 

\crefname{enumi}{}{}
\crefformat{enumi}{#2#1#3}
\Crefformat{enumi}{#2#1#3}

\crefname{assumption}{assumption}{Assumptions}
\crefformat{assumption}{#2assumption~#1#3} 
\Crefformat{assumption}{#2Assumption~#1#3} 

\crefname{construction}{construction}{Constructions}
\crefformat{construction}{#2construction~#1#3} 
\Crefformat{construction}{#2Construction~#1#3} 

\crefname{equation}{}{}
\crefformat{equation}{(#2#1#3)} 
\Crefformat{equation}{(#2#1#3)}

\numberwithin{equation}{section}

\theoremstyle{nonumberplain}
\theoremsymbol{\ensuremath{\blacksquare}}

\newtheorem{proof}{Proof}
\newcommand\pf[1]{\newtheorem{#1}{Proof of \Cref{#1}}}

\newcommand\bF{{\mathbb F}}

\newcommand\bH{{\mathbb H}}

\newcommand\bP{{\mathbb P}}

\newcommand\bR{{\mathbb R}}
\newcommand\bS{{\mathbb S}}

\newcommand\bZ{{\mathbb Z}}

\newcommand\cP{{\mathcal P}}

\newcommand\cR{{\mathcal R}}

\newcommand\cZ{{\mathcal Z}}

\DeclareMathOperator{\id}{id}

\DeclareMathOperator{\Hom}{\mathrm{Hom}}

\newcommand{\cat}[1]{\textsc{#1}}

\newcommand{\qedhere}{\mbox{}\hfill\ensuremath{\blacksquare}}

\newcommand{\xrightarrowdbl}[2][]{%
  \xrightarrow[#1]{#2}\mathrel{\mkern-14mu}\rightarrow
}

\title{Manifold pathologies and Baire-1 functions as cohomotopy groups}
\author{Alexandru Chirvasitu}

\begin{document}

\date{}

\newcommand{\Addresses}{{% additional braces for segregating \footnotesize
  \bigskip
  \footnotesize

  \textsc{Department of Mathematics, University at Buffalo}
  \par\nopagebreak
  \textsc{Buffalo, NY 14260-2900, USA}  
  \par\nopagebreak
  \textit{E-mail address}: \texttt{achirvas@buffalo.edu}

  % % \medskip
  % % 
  % % \textsc{Department of Mathematics, INSTITUTION}
  % % \par\nopagebreak
  % % \textsc{ADDRESS}
  % % \par\nopagebreak
  % % \textit{E-mail address}: \texttt{??}
  % % 

}}

\maketitle

\begin{abstract}
  A slight extension of a construction due to Calabi-Rosenlicht (and later Gabard, Baillif and others) produces a typically non-metrizable $n$-manifold $\mathbb{P}$ by gluing two copies of the open upper half-space $\mathbb{H}_{++}$ in $\mathbb{R}^n$ along the disjoint union of the spaces of rays within $\mathbb{H}_{++}$ originating at points ranging over a subset $S\subseteq \mathbb{R}^{n-1}$ of the boundary $\mathbb{R}^{n-1}=\partial\overline{\mathbb{H}_{++}}$. The fundamental group $\pi_1(\mathbb{P})$ is free on the complement $S^{\times}$ of any singleton in $S\ne\emptyset$, and the main result below is that the first cohomotopy group $\pi^1(\mathbb{P})$, regarded as a space of functions $S^{\times}\to \mathbb{Z}$, is precisely the additive group of integer-valued Baire-1 functions on $S^{\times}$. 

  This occasions a detour on characterizations (perhaps of independent interest) of Baire-1 real-valued functions on a metric space $(B,d)$ as various types of non-tangential boundary limits of continuous functions on $B\times \mathbb{R}_{>0}$.
\end{abstract}


\vspace{.5cm}


\tableofcontents

%%%%%%%%%%%%%%%%%%%%%%%%%%%%%%%%%%%%%%%%%%%%%%%%%%%%%%%%%%%%%%%%%%%%%%%%%%%%%
%%%%%%%%%%%%%%%%%%%%%%%%%%%%%%%%%%%%%%%%%%%%%%%%%%%%%%%%%%%%%%%%%%%%%%%%%%%%%
\section*{Introduction}

The {\it Pr\"ufer surface} ($\bP$ throughout the rest of the present section) of \cite[\S 3]{gab_pruf-mfld} appears to have been initially introduced in \cite[\S 3]{zbMATH03084052} (where it is denoted by $S$) as auxiliary in those authors' verification of a conjecture of Bochner's to the effect that there are non-second-countable (so non-metrizable, or equivalently \cite[Corollary 2.4]{MR0776633}, non-{\it paracompact} \cite[Definition 20.6]{wil_top}) complex-analytic manifolds of complex dimension $>1$. Pr\"ufer surfaces and their cousins exemplify various aspects of the rich subject of non-metrizable manifolds: see e.g. the monograph \cite{zbMATH06375777}, \cite{zbMATH06111447,MR4420868,zbMATH06268283,MR2139756,MR4024539,MR4374950,MR0776633,MR4146463,zbMATH03493562} and their many references, etc.

The description of $\bP$ given in \cite[\S 3]{gab_pruf-mfld} is what \Cref{def:prufn} below builds on: 

\begin{itemize}[wide]
\item Consider the open upper half-plane $\bH_{++}\subset \bR^2$, with its usual topology;

\item Supplement $\bH_{++}$ with a $\bR$-parametrized family of boundary lines $\bR$, with the boundary component $\bR_p\cong \bR$ with parameter
  \begin{equation*}
    p\in \bR:=\text{the usual $x$-axis }\{(x,0)\ |\ x\in \bR\}\subset \bR^2
  \end{equation*}
  consisting of the rays in the upper half-plane based at $p$.

\item The topology on $\bR_p$ is the guessable (``angular'') one, with ray nearness measured by angle inclination, whereas a net $(x_{\lambda})\subset \bH_{++}$ converges to a ray $r\in \bR_p$ if it converges to the point $p\in \bR\subset \bR^2$ in the usual sense, and in addition the rays $\overrightarrow{px_{\lambda}}$ converge to $r$ in the aforementioned angular topology of $\bR_p$.

  This thus far will give a 2-manifold $\bP_{\partial}$ with boundary
  \begin{equation}\label{eq:pruf2bdry}
    \partial \bP_{\partial}
    =
    \coprod_{p\in \bR}\bR_p;
  \end{equation}
  here and throughout the paper an {\it $n$-manifold} is a Hausdorff space locally homeomorphic to either $\bR^n$ or the closed upper half-space therein (as in \cite[Introduction]{MR4420868}, for example).
  
\item Finally, $\bP$ is the {\it double} \cite[Example 9.32]{lee_mfld}
  \begin{equation*}
    \bP := 2\bP_{\partial} := \bP_{\partial}\coprod_{\partial \bP_{\partial}}\bP_{\partial}.
  \end{equation*}
\end{itemize}

This produces a connected (real-analytic, even \cite[\S 3]{zbMATH03084052}) boundary-less 2-manifold, whose fundamental group is \cite[Proposition 3]{gab_pruf-mfld} free on continuum many generators, naturally parametrized in the course of that proof by the subset $\bR^{\times}:=\bR\setminus\{0\}$ of the space $\bR$ housing the parameters $p$ of \Cref{eq:pruf2bdry}.

On the same homotopy-theoretic theme, recall the first {\it cohomotopy} \cite[\S VII.1]{hu_ht_1959} or {\it Bruschlinsky group} (\cite[\S II.7]{hu_ht_1959}, \cite[pre \S 1]{zbMATH03565747})
\begin{equation}\label{eq:bruschx}
  \pi^1(X)
  :=
  [X,\bS^1]
  :=
  \text{free homotopy classes of continuous maps }X\to \bS^1
\end{equation}
of a space $X$. There is generally an obvious pairing (or {\it comparison}) morphism \cite[\S VII.1]{hu_ht_1959}
\begin{equation}\label{eq:comp}
  \pi^1(X)
  \xrightarrow{\quad\cat{comp}\quad}
  \Hom(\pi_1(X),\ \bZ\cong \pi_1(\bS^1))
  ,\quad
  (X\text{ path-connected})
\end{equation}
obtained by composing maps $\bS^1\to X$ and $X\to \bS^1$ into loops in $\bS^1$, and for locally path-connected $X$ (such as the manifold $\bP$) $\cat{comp}$ is injective: continuous maps $X\to \bS^1$ trivial on $\pi_1$ lift \cite[Lemma 79.1]{mnk} through the universal cover $\bR\to \bS^1$ and hence homotope to constants. Recalling \cite[Definition 24.1]{kech_descr} that the {\it Baire-1 functions} between metrizable spaces are those which pull back open sets to countable unions of closed sets, one small sample of the main result in \Cref{th:copi1pruf2} below is as follows:

\begin{theoremN}\label{thn:pruf2coh1}
  The Bruschlinsky subgroup
  \begin{equation*}
    \pi^1(\bP)
    \lhook\joinrel\xrightarrow{\quad\cat{comp}\quad}
    \Hom(\pi_1(\bP),\bZ)
    \cong
    \bZ^{\bR^{\times}}
  \end{equation*}
  of the Pr\"ufer surface is precisely the additive group of Baire-1 functions $\bR^{\times}\to \bZ$.  \qedhere  
\end{theoremN}

In the process of unwinding one of the implications (realizing every Baire-1 function as a cohomotopy class) a few alternative characterizations of the Baire-1 property emerge. Once more compressing and trimming for the purpose of illustration (\Cref{th:phiconv,th:phiconvrel}):

\begin{theoremN}\label{thn:ntlim}
  For a function $B\xrightarrow{f}\bR$ on a metric space $(B,d)$ the following conditions are equivalent.
  \begin{enumerate}[(a),wide]
  \item $f$ is Baire-1.

  \item $f$ is the {\it non-tangential limit} along $B\cong B\times\{0\}$ of a continuous function $B\times \bR_{>0}\xrightarrow{\overline{f}}\bR$ in the sense that
    \begin{equation}\label{eq:thn:ntlim:ntlim}
      \overline{f}(b',t)
      \xrightarrow[b'\to b,\ t\to 0]{\quad d(b,b')\le Ct\quad}
      f(b)
      ,\quad
      \forall b\in B
      ,\quad
      \forall C>0.
    \end{equation}

  \item $f$ is the {\it non-tangential limit} along $B\cong B\times\{0\}\subset B'\times \{0\}$ of a continuous function $B'\times \bR_{>0}\xrightarrow{\overline{f}}\bR$ in the sense that \Cref{eq:thn:ntlim:ntlim} holds with $b'$ ranging over some (equivalently, any) metric space $B'$ containing $(B,d)$ isometrically.  \qedhere
  \end{enumerate}  
\end{theoremN}

%%%%%%%%%%%%%%%%%%%%%%%%%%%%%%%%%%%%%%%%%%%%%%%%%%%%%%%%%%%%%%%%%%%%%%%%%%%%%
\subsection*{Acknowledgements}

This work is partially supported by NSF grant DMS-2001128. 

% % %%%%%%%%%%%%%%%%%%%%%%%%%%%%%%%%%%%%%%%%%%%%%%%%%%%%%%%%%%%%%%%%%%%%%%%%%%%%%
% % %%%%%%%%%%%%%%%%%%%%%%%%%%%%%%%%%%%%%%%%%%%%%%%%%%%%%%%%%%%%%%%%%%%%%%%%%%%%%
% % \section{Preliminaries}\label{se:prel}
% %

%%%%%%%%%%%%%%%%%%%%%%%%%%%%%%%%%%%%%%%%%%%%%%%%%%%%%%%%%%%%%%%%%%%%%%%%%%%%%
%%%%%%%%%%%%%%%%%%%%%%%%%%%%%%%%%%%%%%%%%%%%%%%%%%%%%%%%%%%%%%%%%%%%%%%%%%%%%
\section{1-cohomotopy for boundary-less Pr\"ufer manifolds}\label{se:copi1pruf}

In their two-dimensional incarnations, the following constructions recover objects variously referred to as {\it Pr\"ufer surfaces} (so named because \cite{zbMATH02591746} credits Pr\"ufer for the original idea). They are very much in the spirit of other generalizations in the literature, e.g. the {\it Pr\"uferization} procedure of \cite[Appendix B]{MR4374950}.

\begin{definition}\label{def:prufn}
  Having fixed a positive integer $n\ge 2$ throughout, write
  \begin{equation}\label{eq:halfsp}
    \begin{aligned}
      \bH_{++}=\bH_{++}^n
      &:=
        \text{ strict upper half-space }
        \{(x_1\cdots x_n)\in \bR^n\ |\ x_n>0\}
        \quad\text{and}\\
      \bH_{+}=\bH_{+}^n
      &:=
        \text{ closure }\overline{\bH_{++}}\text{ in }\bR^n= \{(x_1\cdots x_n)\in \bR^2\ |\ x_n\ge 0\}.
    \end{aligned}    
  \end{equation}

  \begin{enumerate}[(1),wide]
  \item\label{item:def:prufn:bord} For a subset $S\subseteq \bR^{n-1}\cong \partial \bH_{+}$ the {\it bordered Pr\"ufer $n$-manifold} $\bP\bF_{S,\partial}=\bP\bF^n_{S,\partial}$ is obtained by
    \begin{itemize}[wide]
    \item setting
      \begin{equation*}
        \bP\bF_{S,\partial}
        :=
        \bH_{++}\sqcup \coprod_{p\in F}\cR_p
        ,\quad
        \cR_p:=\left\{\text{rays in $\bH_{++}$ originating at $p$}\right\}
      \end{equation*}
      as a set (the rays have $p$ as an origin, apart from which they lie entirely {\it above} the hyperplane $\bR^{n-1}\cong \partial\bH_{+}$);

    \item topologizing $\bH_{++}$ as usual, with its subspace topology inherited from $\bR^n$;

    \item topologizing each $\cR_p$ as a subspace of the projective space $\bP^{n-1}$ of lines in $\bR^n$ passing through $p$;
      
    \item and having a net $(x_{\lambda})_{\lambda}\subset \bH_{++}$ converge to $r\in \cR_p\subset \bP\bF_{S,\partial}$ (language: $(x_{\lambda})$ {\it converges tangentially} to $r$) precisely when
      \begin{equation*}
        \begin{aligned}
          x_{\lambda}
          &\xrightarrow[\quad\lambda\quad]{} p
            \text{ in the usual plane topology and also}\\
          \overrightarrow{px_{\lambda}}
          &\xrightarrow[\quad\lambda\quad]{} r
            \text{ in $\cR_p$}.
        \end{aligned}
      \end{equation*} 
    \end{itemize}

    $\bP\bF^2_{\partial} = \bP\bF^2_{\bR,\partial}$ is the surface $P_0$ of \cite[\S 3]{gab_pruf-mfld}, $\bP_s$ of \cite[\S 2]{zbMATH06111447} and what \cite[Appendix B]{MR4374950} would denote by $P_H(H)$. More generally, for arbitrary $Y\subseteq \bR$, $\bP\bF^2_{Y,\partial}$ is the $P_Y(H)$ of loc. cit. We drop the `$F$' subscript when it happens to be the entire boundary $\bR^{n-1}$ of $\bH_{+}$. 
    
  \item\label{item:def:prufn:dbl} The {\it (plain, or double, or boundary-less) Pr\"ufer $n$-manifold} $\bP\bF_{S}=\bP\bF^n_{S}$ is the {\it double} $2\bP\bF_{S,\partial}$ obtained as usual (e.g. \cite[Example 3.80]{lee_top-mfld_2e_2011}), by gluing two copies of $\bP\bF_{S,\partial}$ along its boundary:
    \begin{equation*}
      \bP\bF^n_{S}
      =
      \bP\bF_{S,\partial}\coprod_{\partial\bP\bF_{S,\partial}} \bP\bF_{S,\partial}
      =
      \bP\bF_{S,\partial}\coprod_{\bigcup_{p\in S}\cR_p} \bP\bF_{S,\partial}
    \end{equation*}
    (see \Cref{re:netssuffice}). 

    $\bP\bF^2=\bP\bF^2_{\bR}$ is also the surface $S$ (real analytic, as noted there) of \cite[\S 3]{zbMATH03084052}. 
  \end{enumerate}
\end{definition}

\begin{remark}\label{re:netssuffice}
  Per the usual correspondence (e.g. \cite[Problem 11D]{wil_top}) between net convergence and topology \Cref{def:prufn} gives sufficient information, but \cite[\S 3]{gab_pruf-mfld} actually describes a local neighborhood basis of a ray $r\in \cR_p$ in $\bP\bF^2_{\partial} = \bP\bF^2_{\bR,\partial}$: the rays in $\cR_p$ leaning away from $r$ at an angle $<\varepsilon$ together with the points in $\bH_{++}$ within that ray wedge and less than $\varepsilon$ away from $p$ in the usual Euclidean distance.

  The same source also argues that $\bP\bF^2_{\partial}$ is (as the naming suggests) a topological 2-manifold (surface) with boundary, and hence its double $\bP\bF$ is a boundary-less surface.

  All of this transports immediately to the higher-dimensional, arbitrary-$F$ setup, making
  \begin{itemize}[wide]
  \item all $\bP\bF^n_{S,\partial}$ (as \Cref{def:prufn} suggests) into contractible $n$-manifolds with boundary
    \begin{equation*}
      \partial \bP\bF^n_{S,\partial}
      =
      \bigcup_{p\in S}\cR_p
    \end{equation*}
    so that
    \begin{equation*}
      \partial \bP\bF^n_{S,\partial}\ne \emptyset
      \iff
      S\ne \emptyset;
    \end{equation*}
    
  \item and all doubles $\bP\bF^n_S$ into boundary-less $n$-manifolds, (path-)connected precisely when $S$ is non-empty. 
  \end{itemize}
\end{remark}

\cite[Proposition 3]{gab_pruf-mfld} (attributed there to Baillif) describes of the fundamental group $\pi_1(\bP\bF^2)$ via (one version of) the {\it Seifert-Van Kampen theorem} \cite[Theorem IV.2.2]{massey_basic-alg-top_1991}.

\begin{notation}\label{not:alphapq}
  Recalling \Cref{eq:halfsp}, fix $p,q\in \bR^{n-1}\cong \partial\bH_{+}$. In the context of \Cref{def:prufn}, define a loop $\alpha_{p,q}$ lying in any $\bP\bF^n_{S}$ so long as $p,q\in F$ as follows:
  \begin{itemize}[wide]
  \item if $p=q$ then the loop is constant at the vertical ray through $p=q$;
  \item otherwise, it consists first of a path $\alpha^+_{p,q}$ in $\bP\bF_{S,\partial}$ from the vertical ray at $p$ to that at $q$, traversing the semi-circle in $\bH_{+}$ having the segment $\overline{pq}$ as its diameter and passing through the vertical ray at $\frac {p+q}2$;
  \item followed by the reversal $\alpha^-_{p,q}$ of that path along the mirror-image circle in the second (doubled) copy of $\bH_{+}$.
  \end{itemize}  
\end{notation}

The aforementioned \cite[Proposition 3]{gab_pruf-mfld}, paraphrased, gives an isomorphism
\begin{equation*}
  \begin{tikzpicture}[>=stealth,auto,baseline=(current  bounding  box.center)]
    \path[anchor=base] 
    (0,0) node (l) {$\bR^{\times}$}
    +(3,.5) node (u) {$\text{free group }F_{\bR^{\times}}$}
    +(6,0) node (r) {$\pi_1(\bP\bF^2)$.}
    ;
    \draw[right hook->] (l) to[bend left=6] node[pos=.5,auto] {$\scriptstyle $} (u);
    \draw[->] (u) to[bend left=6] node[pos=.5,auto] {$\scriptstyle \cong$} (r);
    \draw[->] (l) to[bend right=6] node[pos=.5,auto,swap] {$\scriptstyle t \mapsto \text{class of $\alpha_{0,t}$}$} (r);
  \end{tikzpicture}
\end{equation*}

That argument, too, goes through with only the obvious modifications; we append a proof (in slightly different phrasing) for some semblance of completeness and the reader's convenience. 

\begin{lemma}\label{le:whatispi1}
  For any $n\ge 2$, non-empty $S\subseteq \bR^{n-1}\cong \partial \bH_{+}$ and $p_0\in S$ the diagram
  \begin{equation*}
    \begin{tikzpicture}[>=stealth,auto,baseline=(current  bounding  box.center)]
      \path[anchor=base] 
      (0,0) node (l) {$S^{\times}$}
      +(-.6,0) node (ll) {$=:$}
      +(-1.6,0) node (ll) {$S\setminus\{p_0\}$}
      +(3,.5) node (u) {$\text{free group }F_{S^{\times}}$}
      +(6,0) node (r) {$\pi_1(\bP\bF^n_S)$.}
      ;
      \draw[right hook->] (l) to[bend left=6] node[pos=.5,auto] {$\scriptstyle $} (u);
      \draw[->] (u) to[bend left=6] node[pos=.5,auto] {$\scriptstyle \cong$} (r);
      \draw[->] (l) to[bend right=6] node[pos=.5,auto,swap] {$\scriptstyle p \mapsto \text{class of $\alpha_{p_0,p}$}$} (r);
    \end{tikzpicture}
  \end{equation*}
  describes the fundamental group of $\bP\bF^n_{S}$.
\end{lemma}
\begin{proof}
  Consider the two copies $\bH^{\bullet}_{++}$, $\bullet\in\{\uparrow,\downarrow\}$ of $\bH_{++}\subset \bR^n$ glued along $\partial \bP\bF_{S,\partial}$. The open subsets
  \begin{equation*}
    U_p:=\bH_{++}^{\uparrow}\sqcup \cR_p\sqcup \bH_{++}^{\downarrow}
    ,\quad
    p\in S.
  \end{equation*}
  are contractible, and intersect pairwise along $\bH_{++}^{\uparrow}\sqcup \bH_{++}^{\downarrow}$. In its {\it groupoid}-theoretic phrasing (e.g. \cite[\S 2, Theorem]{br_vk}), the Van Kampen theorem identifies the groupoid $\pi_1(X,\left\{\uparrow,\downarrow\right\})$ attached (\cite[Chapter 6]{higg_gpds}, \cite[\S 2]{zbMATH03241070}) to any choice
  \begin{equation*}
    \uparrow\in \bH_{++}^{\uparrow}
    \quad\text{and}\quad
    \downarrow\in \bH_{++}^{\downarrow}
  \end{equation*}
  with that of a family of segments $I_p$, $p\in S$ glued along a common pair of endpoints (and given its CW topology). After contracting one segment, this is nothing but a $S^{\times}$-indexed bouquet of circles, hence the conclusion \cite[Example 1.21]{hatcher}. 
\end{proof}

Consider, now, for some connected $\bP\bF=\bP\bF^n_S$, the first cohomotopy group $\pi^1(\bP\bF)$ of \Cref{eq:bruschx}. In reference to \Cref{th:copi1pruf2}, recall also the following two properties one frequently considers for functions between topological (usually at least metric) spaces:
\begin{itemize}[wide]
\item the {\it Baire class one} \cite[Definition 24.1]{kech_descr} (or {\it Baire one}, or  {\it Baire-1} here, for short) are those through which the preimage of every open set is {\it $F_{\sigma}$}, i.e. \cite[\S 1.A, p.1]{kech_descr} a countable union of closed sets;

\item the {\it Baire} functions of \cite[post Example 3.1.31]{sriv_borel} (henceforth {\it sequentially Baire} here, for clarity, when there is danger of confusion) are the sequential pointwise limits of continuous functions. 
\end{itemize}
The two properties are often equivalent and conflated, and the terminology reflects this: \cite[\S 7]{oxt_meas-cat_2e_1980}, for instance, {\it defines} `first class of Baire' as `(sequentially) Baire', so does \cite[Introduction]{steg_b1}, etc.

Sequentially Baire implies Baire-1 (e.g. \cite[proof of Theorem 4, pp.986-987]{steg_b1}), but the converse does not hold unconditionally \cite[paragraph preceding Theorem 24.10]{kech_descr}. The two properties are indeed equivalent for real-valued functions on metric spaces (\cite[Theorem 24.10]{kech_descr} assumes separability for the domain, but that condition does not seem necessary) or for Banach-space-valued functions on complete metric spaces \cite[Theorem 4]{steg_b1}.

% % \begin{theorem}\label{th:copi1pruf2}
% %   The comparison embedding
% %   \begin{equation*}
% %     \pi^1(\bP\bF)
% %     \lhook\joinrel\xrightarrow{\quad\cat{comp}\quad}
% %     \Hom(\pi_1(\bP\bF),\ \bZ)
% %     ,\quad
% %     \bP\bF=\text{Pr\"ufer surface of \Cref{def:prufn}\Cref{item:def:prufn:dbl}}
% %   \end{equation*}
% %   identifies $\pi^1(\bP\bF)$, as a subspace of
% %   \begin{equation*}
% %     \Hom(\pi_1(\bP\bF),\bZ)
% %     \cong
% %     \Hom(F_{\bR^{\times}},\bZ)
% %     \cong
% %     \left(\cat{functions }\bR^{\times}\to \bZ\right),
% %   \end{equation*}
% %   with the additive group of Baire-1 functions $\bR^{\times}\to \bZ$. 
% % \end{theorem}

\begin{theorem}\label{th:copi1pruf2}  
  Fix a positive integer $n\ge 2$ and a non-empty subspace $S\subseteq \bR^{n-1}\cong \partial\bH_+^n$, and set $S^{\times}:=S\setminus\{p_0\}$ for fixed $p_0\in S$. 

  The comparison embedding
  \begin{equation*}
    \pi^1(\bP\bF)
    \lhook\joinrel\xrightarrow{\quad\cat{comp}\quad}
    \Hom(\pi_1(\bP\bF),\ \bZ)
    ,\quad
    \bP\bF:=\text{Pr\"ufer $n$-fold $\bP\bF^n_S$ of \Cref{def:prufn}\Cref{item:def:prufn:dbl}}
  \end{equation*}
  identifies $\pi^1(\bP\bF)$, as a subspace of
  \begin{equation*}
    \Hom(\pi_1(\bP\bF),\bZ)
    \cong
    \Hom(F_{S^{\times}},\bZ)
    \cong
    \left(\cat{functions }S^{\times}\to \bZ\right),
  \end{equation*}
  with the additive group of Baire-1 functions $S^{\times}\to \bZ$. 
\end{theorem}
\begin{proof}[$\xRightarrow{\quad}$]
  
  The fact that the group structure on $\pi^1$ transports over to the usual additive structure on $\bZ$-valued functions is of course not at issue; the substance of the claim is the link to the Baire property. Here, we argue that any morphism $\pi_1(\bP\bF)\to \bZ$ arising from the cohomotopy class of a continuous map $\bP\bF\xrightarrow{g}\bS^1$ is Baire-1 when restricted to the family $S^{\times}$ of generators $\alpha_{p_0,p}$, $p\in S^{\times}$. 

  $\bH_{++}$ can be identified with the hyperbolic space $\bH^n$ in its {\it upper half-plane model} \cite[\S 2.3]{thurst_3mfld-unpublished}, whereupon the semicircles $\alpha_{p_0,p}^+$ of \Cref{not:alphapq} are precisely the geodesics connecting $p_0$ and $p$ regarded as points in the {\it sphere (or boundary) at infinity} (\cite[\S 2.1]{thurst_3mfld-unpublished}, \cite[Chapter II]{ball_nonpos})
  \begin{equation*}
    \partial_{\infty}\bH^n
    \cong
    \bS^{n-1}
    \cong
    \bR^{n-1}\sqcup \{\infty\}
    \cong \partial \bH_+\sqcup \{\infty\}
  \end{equation*}
  of $\bH^n$. Upon effecting an automorphism of $\bH^n$, we can assume $p_0$ itself is the ideal point $\infty\in \bS^{-1}$, so that the half-circles $\alpha_{p_0,p}^+$ become vertical lines orthogonal to $\bR^{n-1}$. It will be convenient to work in this parametrization, so that
  \begin{itemize}[wide]
  \item $g$ restricts to a continuous function $S^{\times}\times \bR_{>0}\xrightarrow{g}\bS^1$;

  \item converging to $1\in \bS^1$ as height increases; 
    
  \item and furthermore extending continuously to the boundary $(p,0)$ along every vertical half-line $\{p\}\times \bR_{\ge 0}$, $p\in S^{\times}$.
  \end{itemize}
  The vertical path $\infty\xrightarrow{\pi^+} (p,0)$ is $\alpha^+_{p_0,p}$, and said extension is the restriction
  \begin{equation*}
    g_p^+:=g|_{\infty\xrightarrow{\pi^+}(p,0)}
  \end{equation*}
  of $g$ to it. All of this also applies to the second copy of $\bH_{++}$ constituting $\bP\bF$: $g$, initially defined on all of $\bP\bF$, also restricts the semicircle $\alpha^-_{p_0,p}$ recast as a separate (and reversed) copy $(p,0)\xrightarrow{\pi^-}\infty$ of the path $\pi^+$, and the loop $g_*(\alpha_{p_0,p})$ in $\bS^1$ based at $1\in \bS^1$ is obtained as the concatenation $g_p^+*g_p^-$ of 
  \begin{equation}\label{eq:gpm}
    g_p^+:=
    g|_{\infty\xrightarrow{\pi^+}(p,0)}
    \quad\text{and}\quad
    g_p^-:=
    g|_{(p,0)\xrightarrow{\pi^-}\infty}
    ,\quad
    p\in S^{\times}. 
  \end{equation}
  Now, $g^+_p$, regarded as a continuous map
  \begin{equation*}
    I:=[I_{\ell}, I_r]:=\overline{\infty\ (p,0)}
    \xrightarrow{\quad}
    \bS^1,
  \end{equation*}
  lifts to a continuous map $I\xrightarrow{h_p^+}\bR$ with $h_p^+(I_{\ell})=0$ (by homotopy lifting \cite[lemma 79.1]{mnk}). The reversed $g_p^-$ similarly lifts to a map $I\xrightarrow{h_p^-}\bR$ with $h_p^-(I_{\ell})=0$ and
  \begin{equation*}
    \text{class }[g_*(\alpha_{p_0,p})]
    \in \pi_1(\bS^1)\cong \bZ
  \end{equation*}
  is nothing but the (automatically integral) difference $h_p^-(I_r)-h_p^+(I_r)$. To conclude, simply note that because $g$ is continuous on both copies of $\bH_{++}$ and it extends continuously along the vertical paths $\pi^{\pm}$ of \Cref{eq:gpm}, the maps
  \begin{equation*}
    S^{\times}\ni p
    \xmapsto{\quad}
    h_p^{\pm}(I_r)
    \in \bR
  \end{equation*}
  are both sequential limits of continuous real-valued functions and hence \cite[\S 24.B, $2^{nd}$ paragraph]{kech_descr} Baire-1. 
\end{proof}

%%%%%%%%%%%%%%%%%%%%%%%%%%%%%%%%%%%%%%%%%%%%%%%%%%%%%%%%%%%%%%%%%%%%%%%%%%%%%
%%%%%%%%%%%%%%%%%%%%%%%%%%%%%%%%%%%%%%%%%%%%%%%%%%%%%%%%%%%%%%%%%%%%%%%%%%%%%
\section{Non-tangential limits and Borel-1 functions}\label{se:nontg}

Let $(B,d)$ be a metric space (frequently {\it Polish}, i.e. \cite[Definition 3.1]{kech_descr} complete and separable), and consider functions $B\xrightarrow{f}\bR$ exhibiting two contrasting types of behavior.

\begin{itemize}[wide]
\item $f$ might be extensible along
  \begin{equation*}
    B\cong B\times\{0\}\subset B\times \bR_{\ge 0}
  \end{equation*}
  to a continuous function $B\times \bR_{\ge 0}\xrightarrow{\overline{f}}\bR$; naturally, $f$ itself must then be continuous.

\item Alternatively, $f$ might be recoverable as the {\it radial limit} (in terminology borrowed from \cite[\S 11.5]{rud_rc_3e_1987}, say)
  \begin{equation}\label{eq:radlim}
    f(b) = \lim_{t\searrow 0}\overline{f}(b,t)
    ,\quad
    B\times \bR_{>0}\xrightarrow[\text{continuous}]{\quad \overline{f}\quad} \bR.
  \end{equation}
  It is an easy check that this is possible {\it precisely} when $f$ is a pointwise limit of a sequence of continuous functions $B\xrightarrow{f_n}\bR$, i.e. Baire-1. 
  
  % % In order to have the various competing imaginable definitions equivalent (\cite[Theorems 24.10 and 24.15]{kech_descr}) we only consider these over Polish metric spaces, as warned above. 
  % % 
\end{itemize}

The tangential convergence of \Cref{def:prufn}\Cref{item:def:prufn:bord} suggests what looks like intermediate-type behavior: $f\in \bR^B$ might be the {\it non-tangential} limit at $t\searrow 0$ of a continuous function $\overline{f}$: in place of \Cref{eq:radlim}, impose the stronger requirement
\begin{equation}\label{eq:ntlim}
  f(b) = \lim_{(b',t)\xrightarrow[\text{within any $W_{b,C}$}]{\quad}(b,0)}\overline{f}(b,t)
  ,\quad
  B\times \bR_{>0}\xrightarrow[\text{continuous}]{\quad \overline{f}\quad} \bR,
\end{equation}
where $W_{b,C}$, $C>0$ is the wedge
\begin{equation}\label{eq:wbc}
  W_{b,C}:=\left\{(b',t)\in B\times \bR_{>0}\ |\ d(b,b')\le Ct\right\}.
\end{equation}
The language is again borrowed from complex/harmonic analysis (\cite[\S 11.18]{rud_rc_3e_1987}, \cite[\S 3]{zbMATH03983594}, \cite[\S 1]{zbMATH00001169}, etc.), where one often considers such limits at points on the unit circle of functions defined in the open unit disk. 

In the current setup, though, the $W_{b,C}$ of \Cref{eq:wbc} have less to do with the geometry of tangents than with bounding, in a controlled fashion, the distance $d(b',b)$ of the first coordinate of $(b',t)$ from $b$ in terms of the second coordinate. The following gadgetry is intended to achieve that same purpose.

\begin{definition}\label{def:smlgerm}
  \begin{enumerate}[(1),wide]
  \item\label{item:def:smlgerm:smgg} A {\it (smallness) gauge} is a {\it germ} at $0$ of non-negative, 0-vanishing continuous functions $\bR_{\ge 0}\to \bR_{\ge 0}$ in the usual sheaf-theoretic sense of the term `germ' (\cite[\S I.1, p.2]{bred_shf_2e_1997}, \cite[Terminology 2.1.2]{ks_shv-mfld}): an equivalence class of continuous
    \begin{equation*}
      [0,\varepsilon)
      \xrightarrow{\quad\phi\quad}
      \bR_{\ge 0}
      ,\quad
      \phi(0)=0
    \end{equation*}
    with two such functions declared equivalent if they are both defined and agree on some sufficiently small $[0,\varepsilon')$. 
    
    There is an obvious ordering on the collection $\cat{Gg}$ of smallness gauges: $\phi\le \phi'$ if $\phi(x)\le \phi'(x)$ for all sufficiently small $x\in \bR_{\ge 0}$ (notation: $x\sim 0$).

  \item\label{item:def:smlgerm:bdrylim} Let $\Phi=\{\phi\}$ be a family of gauges and $(B,d)$ a metric space. The function $B\times \bR_{>0}\xrightarrow{\overline{f}} \bR$ {\it has $\Phi$-(boundary-)limit} (or {\it $\Phi$-(boundary-)converges to}) a function
    \begin{equation*}
      B\times \bR_{\ge 0} \supset B\times\{0\} \cong B
      \xrightarrow{\quad f\quad}
      \bR
    \end{equation*}
    if
    \begin{equation*}
      f(b) = \lim_{(b',t)\xrightarrow[\text{within $W_{b,\phi}$}]{\quad}(b,0)}\overline{f}(b',t)
      ,\quad
      \forall b\in B
      ,\quad
      \forall \phi\in \Phi
    \end{equation*}
    for the analogue
    \begin{equation}\label{eq:wbphi}
      W_{b,\phi}:=\left\{(b',t)\in B\times \bR_{>0},\ t\sim 0\ |\ d(b,b')\le\phi(t)\right\}
    \end{equation}
    of \Cref{eq:wbc}.

  \item\label{item:def:smlgerm:phiconvseq} Similarly, a sequence $(f_{n})_{n}$ of (typically continuous) functions $B\xrightarrow{f_{n}}\bR$ is said to {\it have $\Phi$-limit} (or {\it $\Phi$-converge to}) $B\xrightarrow{f}\bR$ if
    \begin{equation}\label{eq:def:smlgerm:phiconvseq}
      f(b) = \lim_{\left(b',\frac 1n\right)\xrightarrow[\text{within $W_{b,\phi}$}]{\quad}(b,0)}f_n(b')
      ,\quad
      \forall b\in B
      ,\quad
      \forall \phi\in \Phi.
    \end{equation}
  \end{enumerate}
  We abbreviate $\{\phi\}$-convergence to $\phi$ convergence for singletons. 
\end{definition}

\begin{theorem}\label{th:phiconv}
  For a function $B\xrightarrow{f}\bR$ on a metric space $(B,d)$ the following conditions are equivalent.
  \begin{enumerate}[(a),wide]
  \item\label{item:th:phiconv:contggfam} $f$ is the $\Phi$-boundary-limit of a continuous function $B\times \bR_{>0}\xrightarrow{\overline{f}}\bR$ for some (all) upper-bounded families $\Phi\subset (\cat{Gg},\le)$. 

  \item\label{item:th:phiconv:contggsing} $f$ is the $\phi$-boundary-limit of a continuous function $B\times \bR_{>0}\xrightarrow{\overline{f}}\bR$ for some (all) $\phi\in \cat{Gg}$. 

  \item\label{item:th:phiconv:seqggfam} $f$ is the $\Phi$-boundary-limit of a sequence $(f_n)$ of continuous functions for some (all) upper-bounded families $\Phi\subset (\cat{Gg},\le)$. 

  \item\label{item:th:phiconv:seqggsing} $f$ is the $\phi$-limit of a sequence $(f_n)$ of continuous functions for some (all) $\phi\in \cat{Gg}$.

  \item\label{item:th:phiconv:b1} $f$ is Baire-1. 
  \end{enumerate}
\end{theorem}
\begin{proof}
  Marking the universal (`all') and existential (`some') with `$\forall$' and `$\exists$' subscripts respectively, $\bullet_{\forall}$ formally imply the respective $\bullet_{\exists}$. Additionally,
  \begin{equation*}
    \Cref{item:th:phiconv:contggfam}_{\bullet}
    \xRightarrow{\quad}
    \Cref{item:th:phiconv:seqggfam}_{\bullet}
    \quad\text{and}\quad
    \Cref{item:th:phiconv:contggsing}_{\bullet}
    \xRightarrow{\quad}
    \Cref{item:th:phiconv:seqggsing}_{\bullet}
    ,\quad
    \bullet\in \{\forall,\ \exists\}
  \end{equation*}
  by setting $f_n:=\overline{f}\left(\bullet,\frac 1n\right)$. Because furthermore $\phi$-convergence persists upon passing to smaller $\phi$, one can substitute an upper bound for upper-bounded $\Phi\subset \cat{Gg}$ to recover the `all' family statements from their respective singleton counterparts. All conditions plainly imply \Cref{item:th:phiconv:b1} because pointwise convergence is simply $0$-convergence, so that
  \begin{equation}\label{eq:th:phiconv:takestock}
    \Cref{item:th:phiconv:contggfam}_{\forall}
    \xLeftrightarrow{\quad}
    \Cref{item:th:phiconv:contggsing}_{\forall}
    \xRightarrow{\quad}
    \left(\text{everything else}\right)
    \xRightarrow{\quad}
    \Cref{item:th:phiconv:b1}.
  \end{equation}
  As it will be convenient to ultimately settle into arguments involving the discrete conditions \Cref{item:th:phiconv:seqggfam} and/or \Cref{item:th:phiconv:seqggsing}, we simplify matters to that extent.  
  
  \begin{enumerate}[(I),wide]
  \item {\bf : $\Cref{item:th:phiconv:seqggsing}_{\forall}$ $\xRightarrow{\quad}$ $\Cref{item:th:phiconv:contggsing}_{\forall}$ } Fix an arbitrary gauge $\phi$, as required by $\Cref{item:th:phiconv:contggsing}_{\forall}$. Then pick a larger gauge $\phi'\ge \phi$ with the property that
    \begin{equation}\label{eq:condphi'}
      \min_{I_n}\phi'
      >
      \max_{I_n}\phi
      ,\quad\forall n\gg 0
      \quad\text{with}\quad
      I_n:=\left[\frac 1{n+1},\ \frac 1n\right],
    \end{equation}
    and a $\phi'$-convergent sequence $f_n\xrightarrow[n]{\ }f$ afforded by $\Cref{item:th:phiconv:seqggsing}_{\forall}$. Finally, set $\overline{f}\left(\bullet,\frac 1n\right):=f_n$ and thence interpolate linearly:
    \begin{equation*}
      \begin{aligned}
        \overline{f}\left(\bullet,\ \frac {s}{n+1}+\frac{1-s}{n}\right)
        &:=
          s\overline{f}\left(\bullet,\frac {1}{n+1}\right)
          +(1-s)\overline{f}\left(\bullet,\frac{1}{n}\right)\\
        &=
          sf_{n+1}+(1-s)f_n
          ,\quad s\in [0,1]. 
      \end{aligned}      
    \end{equation*}
    By the very choice of $\phi'$, assuming $n\gg 0$ and $t\in I_n$ (in the notation of \Cref{eq:condphi'}),
    \begin{equation*}
      (b',t)\in W_{b,\phi'}
      \xRightarrow{\quad}
      \left(b',\frac 1{n+1}\right)
      ,\
      \left(b',\frac 1{n}\right)
      \in W_{b,\phi}.
    \end{equation*}
    But then for $n\gg 0$ both $f_{n+1}(b')$ and $f_n(b')$ will be close to $f(b)$, hence so will their convex combination $f(b')$.
    
    Given \Cref{eq:th:phiconv:takestock}, the following implication will close the circle and complete the proof.

  \item {\bf : $\Cref{item:th:phiconv:b1}$ $\xRightarrow{\quad}$ $\Cref{item:th:phiconv:seqggsing}_{\forall}$ } Recall once more \cite[Theorem 24.10]{kech_descr}: Baire-1 real-valued functions on metric spaces are Baire (i.e. sequential limits of continuous functions); the assumed separability of the domain in loc. cit. is inessential.

    In order to achieve \Cref{eq:def:smlgerm:phiconvseq} for such $f$ we need $|f(b) - f_n(b')|$ appropriately small. We have an estimate
    \begin{equation*}
      \begin{aligned}
        |f(b) - f_n(b')|
        &\le |f(b) - f_n(b)| + |f_n(b) - f_n(b')|\\
        &\le |f_n(b)-f(b)| + \mathrm{dil}(f_n)\cdot d(b,b')\\
        &\le |f_n(b)-f(b)| + \mathrm{dil}(f_n)\cdot \phi\left(\frac 1n\right)
          \quad\text{whenever }\left(b',\frac 1n\right)\in W_{b,\phi},
      \end{aligned}      
    \end{equation*}
    where
    \begin{equation*}
      \mathrm{dil}(\psi)
      :=
      \sup_{x\ne x'}\frac{d_Y(\psi x,\psi x')}{d_X(x,x')}
      \in \bR_{\ge 0}\sqcup \{\infty\}
      ,\quad
      (X,d_X)\xrightarrow{\quad\psi\quad}(Y,d_Y)
    \end{equation*}
    denotes the {\it dilatation} (\cite[Definition 1.1]{grom_metr-struct_2007}, \cite[Definition 1.4.1]{bbi}) of a function between two metric spaces. It will be enough, then, in selecting a sequence $(f_n)$ of continuous functions converging pointwise to $f$, to also ensure that
    \begin{equation}\label{eq:smalldilrepl}
      \mathrm{dil}(f_n)\cdot \phi\left(\frac 1n\right)
      \xrightarrow[\quad n\to \infty\quad]{\quad}
      0.
    \end{equation}
    Because $a_n:=\phi(1/n)$ converges to 0, it will furthermore be enough to have $g_n$ converging pointwise to $f$ and each with finite dilatation (i.e. have $g_n$ {\it Lipschitz} \cite[Definition 1.1]{grom_metr-struct_2007}). For given that, we can produce the desired sequence $(f_n)$ by repeating each $g_i$ sufficiently so as to ensure the dilatations increase as slowly as \Cref{eq:smalldilrepl} requires:
  \begin{itemize}[wide]
  \item Set the initial segment
    \begin{equation*}
      f_n = g_1,\quad 0<n\le m_1
      \quad\text{with}\quad
      \mathrm{dil}(g_2)\cdot a_m < 1,\quad\forall m>m_1.
    \end{equation*}

  \item Then, similarly,
    \begin{equation*}
      f_n = g_2,\quad m_1<n\le m_2
      \quad\text{with}\quad
      \mathrm{dil}(g_3)\cdot a_m < \frac 12,\quad\forall m>m_2.
    \end{equation*}

  \item And it should be clear how the process continues.     
  \end{itemize}
  As for the last step of producing a Lipschitz sequence converging to $f$, this is well-known to be possible for Baire functions: it is remarked in the course of the proof of \cite[Theorem 6]{alp_ext-diff}, for instance, that this follows from \cite[Proposition 3.9]{cl_baire}. 
  \end{enumerate}
\end{proof}

\begin{remarks}\label{res:cl_baire_p39}
  \begin{enumerate}[(1),wide]
  \item\label{item:res:cl_baire_p39:initelab} Some small elaboration might be of use to those readers who (like this one) find, on first perusal, the reference to \cite[Proposition 3.9]{cl_baire} in the proof of \cite[Theorem 6]{alp_ext-diff} somewhat cryptic.

    Consider a (real) {\it unital linear lattice} $\Phi$ of functions $X\to \bR$ on a set $X$:
    \begin{itemize}[wide]
    \item $1\in \Phi$ (unitality);

    \item and a linear space under the usual additive/scaling;

    \item and a lattice under the standard ordering, i.e. closed under binary (or finite) maxima and minima.
    \end{itemize}
    Coupled, \cite[Propositions 3.1 and 3.9]{cl_baire} then show that the class $\Phi^p$ of sequential pointwise limits of $\Phi$-members depends only on the collection
    \begin{equation*}
      \cZ(\Phi):=
      \left\{f^{-1}(0)\ |\ f\in \Phi\right\}\subseteq 2^X
    \end{equation*}
    of $\Phi$-member zero-sets. Now simply note that for metric spaces $(X,d)$ continuous and Lipschitz real-valued functions (both linear lattices containing the constants!) have the same zero-sets: precisely the closed subsets of $X$. Indeed, every closed $Z\subseteq X$ is the vanishing locus of the (1-Lipschitz) distance function
    \begin{equation*}      
      X\ni x
      \xmapsto{\quad}
      d(x,Z):=\inf\{d(x,y)\ |\ y\in Z\}
      \in \bR_{\ge 0}
    \end{equation*}
    attached to $Z$.

  \item In fact, more is true, per the selfsame \cite[Proposition 3.9]{cl_baire}: $\Phi^p$ depends not, strictly speaking, on $\cZ(\Phi)$, but rather only on the formally less informative collection $\cZ(\Phi)^{c\sigma\delta}$ defined by chaining the operators
    \begin{equation*}
      \begin{aligned}
        \text{class of sets }\cP
        &\xmapsto{\quad}\cP^c
          :=\left\{\text{complements of $Z\in \cP$}\right\}\\
        \cP
        &\xmapsto{\quad}\cP^{\sigma}
          :=\left\{\text{countable unions of $\cP$-members}\right\}
          \quad\text{and}\\
        \cP
        &\xmapsto{\quad}\cP^{\delta}
          :=\left\{\text{countable intersections of $\cP$-members}\right\}.
      \end{aligned}
    \end{equation*}
    Given more structure on $X$, the same principle would allow recovering Baire-1 functions as pointwise limits of even more restrictive or better-behaved classes of functions. If $X$ is a smooth manifold, for instance, the unital linear lattice of smooth functions also recovers precisely the closed subsets of $X$ as zero-sets \cite[Theorem 2.29]{lee_mfld} (a result frequently attributed to Whitney, as in \cite[\S 1, first paragraph]{zbMATH05358251}).
  \end{enumerate}
\end{remarks}

For the purpose of applying all of this back to Pr\"ufer manifolds, it will be useful to also have ``relative'' versions handy for the various properties listed in \Cref{th:phiconv}, all equivalent to their ``absolute'' counterparts.

\begin{definition}\label{def:relb1}
  Consider a subset $B\lhook\joinrel\xrightarrow{\iota}(B',d)$ of a metric space.
  \begin{enumerate}[(1),wide]
  \item A function $B\xrightarrow{f}\bR$ is said to satisfy the condition \Cref{item:def:smlgerm:bdrylim} or \Cref{item:def:smlgerm:phiconvseq} of \Cref{def:smlgerm} {\it $\iota$-relatively} (or {\it relatively to $\iota$}) if the functions $\overline{f}$ (the sequence $(f_n)$) with the requisite properties is definable on the larger $B'\times \bR_{>0}$ (respectively $B'$).

  \item The same terminology applies to Baire functions: $B\xrightarrow{f}\bR$ is {\it $\iota$-relatively} Baire-1 if it is a sequential pointwise limit of continuous functions defined globally on $B'\supset B$. 
  \end{enumerate}
\end{definition}

The preceding terms are convenient to have, and use, but only provisionally:

\begin{theorem}\label{th:phiconvrel}
  Let $B\lhook\joinrel\xrightarrow{\iota} (B',d)$ be a subset of a metric space, correspondingly regarded as a metric space $(B',d)$ in its own right.

  For a function $B\xrightarrow{f}\bR$ the conditions of \Cref{th:phiconv} are all equivalent to their $\iota$-relative counterparts in the sense of \Cref{def:relb1}. 
\end{theorem}
\begin{proof}
  That every $\iota$-relative condition implies the $\iota$-relative Baire-1 property goes through as before, and the latter's equivalence to plain Baire-1 follows from \cite[Proposition 3.9]{cl_baire}, as recalled in \Cref{res:cl_baire_p39}\Cref{item:res:cl_baire_p39:initelab}. Indeed, although in general not all continuous functions on $B'$ extend to $B$, the two classes of functions $B'\to R$ (plain continuous and extensible ones) have the same zero-sets: the closed subsets of $(B',d)$.
\end{proof}

In particular, going back to the non-tangential ``approach domains'' \Cref{eq:wbc}:

\begin{corollary}\label{cor:approachnontg}
  A function $(B,d)\xrightarrow{f}\bR$ on a metric space is Baire-1 precisely when it is $\iota$-relatively the non-tangential limit of a continuous function $B'\times \bR_{>0}\xrightarrow{\overline{f}}\bR$ in the sense of \Cref{eq:ntlim} for
  \begin{itemize}[wide]
  \item some
  \item or equivalently, every isometric embedding $(B,d)\xrightarrow{\iota}(B',d)$. 
  \end{itemize}
\end{corollary}
\begin{proof}
  This is the equivalence \Cref{item:th:phiconv:b1} $\iff$ \Cref{item:th:phiconv:contggfam} of \Cref{th:phiconv}, along with its relative version in \Cref{th:phiconvrel}, applied to the family
  \begin{equation*}
    \Phi = \left\{x\mapsto Cx\ |\ C>0\right\}\subset \cat{Gg} 
  \end{equation*}
  of smallness gauges: simply note that the family is indeed upper-bounded, for instance by (the continuous extension at 0 of) $x\mapsto \frac 1{\ln(1/x)}$.
\end{proof}

We are now ready for the outstanding implication in \Cref{th:copi1pruf2}.

\pf{th:copi1pruf2}
\begin{th:copi1pruf2}[$\xLeftarrow{\quad}$]

  Extend a Baire-1 function $S^{\times}\xrightarrow{f}\bZ$ to $S=S^{\times}\sqcup\{p_0\}$ by $p_0\mapsto 0$ (still Baire-1), and realize that, per \Cref{cor:approachnontg}, as a non-tangential limit \Cref{eq:radlim} (with $B:=S$) for a continuous real-valued function $\overline{f}$ defined over $\bH_{++}$. The imposed non-tangential-limit property then extends $\overline{f}$ across the boundary
  \begin{equation}\label{eq:cuprp}
    \partial \bP\bF_{S,\partial}
    =
    \coprod_{p\in S}\cR_p
    \quad\text{of}\quad
    \bP\bF_{S,\partial}
  \end{equation}
  to a continuous function (denoted economically by the same symbol) $\bP\bF_S\xrightarrow{\overline f}\bR$, taking integer values on \Cref{eq:cuprp}, and in particular collapsing every individual ray bouquet $\cR_p$ (for every $p\in S$) to the single integer $\overline{f}(p)=f(p)\in \bZ$. 

  Finally, the original $S^{\times}\xrightarrow{f}\bZ$ will be recovered as the map $\pi_1(\bP\bF_S)\xrightarrow{g_*} \pi_1(\bS^1)$ for the continuous map $\bP\bF_S\xrightarrow{g}\bS^1$ defined by
    \begin{equation*}
      g(x)=
      \begin{cases}
        \exp\left(2\pi i \overline{f}(x)\right)
        \in \bS^1
        &\quad\text{for $x$ in the first copy of }\bH_{++}\\
        \exp\left(-2\pi i \overline{f}(x)\right)
        \in \bS^1
        &\quad\text{for $x$ in the second copy of }\bH_{++}\\
      \end{cases}
    \end{equation*}
    and, naturally (and compatibly by continuity), $g(x)=1\in \bS^1$ for $x$ in \Cref{eq:cuprp}.
\end{th:copi1pruf2}

\addcontentsline{toc}{section}{References}
%\bibliography{bib}{}
%\bibliographystyle{plain}

\def\polhk#1{\setbox0=\hbox{#1}{\ooalign{\hidewidth
  \lower1.5ex\hbox{`}\hidewidth\crcr\unhbox0}}}
  \def\polhk#1{\setbox0=\hbox{#1}{\ooalign{\hidewidth
  \lower1.5ex\hbox{`}\hidewidth\crcr\unhbox0}}}
  \def\polhk#1{\setbox0=\hbox{#1}{\ooalign{\hidewidth
  \lower1.5ex\hbox{`}\hidewidth\crcr\unhbox0}}}
  \def\polhk#1{\setbox0=\hbox{#1}{\ooalign{\hidewidth
  \lower1.5ex\hbox{`}\hidewidth\crcr\unhbox0}}}
  \def\polhk#1{\setbox0=\hbox{#1}{\ooalign{\hidewidth
  \lower1.5ex\hbox{`}\hidewidth\crcr\unhbox0}}}

\Addresses

\end{document}